\documentclass[11pt]{article}
\usepackage{amssymb,amsthm}
\usepackage[margin=1in]{geometry}
\usepackage{indentfirst}
\usepackage{amsmath}
\newtheorem{theorem}{Theorem}[section]
\newtheorem{lemma}[theorem]{Lemma}

\newtheorem{claim}[theorem]{Claim}
\newtheorem{corollary}[theorem]{Corollary}

\newtheorem{conjecture}[theorem]{Conjecture}

\newcommand{\rhat}{\widehat r}
\newcommand{\sq}{\sqcup}

\title{Size Ramsey numbers of star forests with two components on one side}
\date{}

\author{{Chunlin You\footnote{School of Mathematics and Statistics, Yancheng Teachers University, Yancheng
			224002, P.~R.~China. E-mail: {\tt chunlin\_you@163.com}.
			Supported by the National Natural Science Foundation of China
			(No. 12401469) and the Young Elite Scientist Sponsorship Program by JSAST (JSTJ-2025-892).}}}

\begin{document}
\maketitle

\begin{abstract}
We determine the size Ramsey number of a two-component star forest
versus an arbitrary star forest.  Let $q,a,b,n_1,\ldots,n_q$ be positive
integers with $a\ge b$ and $n_1\ge\cdots\ge n_q$.  We prove that
\[
\begin{split}
\rhat\left(K_{1,a}\sq K_{1,b},
 \bigsqcup_{i=1}^{q}K_{1,n_i}\right)
={}&(a+n_1-1)
+\sum_{i=1}^{q-1}\max\{a+n_{i+1}-1,b+n_i-1\}\\
&+(b+n_q-1).
\end{split}
\]
This proves the conjecture of Burr, Erd\H{o}s, Faudree, Rousseau and Schelp
(Erd\H{o}s Problem~\#561) in the case where one of the two star forests
has exactly two components.  
The formula holds without any restriction on the parity of the star sizes.
The proof combines vertex deletion, iterated matching deletion, and Lovász's factor theorem.

\end{abstract}

\section{Introduction}

Throughout this paper, all graphs are finite and simple.  For graphs
$G,F$, and $H$, we write $G\to(F,H)$ if every red-blue coloring of $E(G)$ contains
a red copy of $F$ or a blue copy of $H$.  Erd\H{o}s, Faudree,
Rousseau and Schelp~\cite{ErdosEtAl1978} introduced the size Ramsey
number of $F$ and $H$, defined by
\[
 \rhat(F,H)=\min\{e(G):G\to(F,H)\}.
\]
Beck~\cite{Beck1983} proved that $\rhat(P_n,P_n)=O(n)$,
where $P_n$ denotes the path on $n$ vertices, answering a question
of Erd\H{o}s. For further results on size Ramsey numbers, see the
survey of Faudree and Schelp~\cite{FaudreeSchelp2002}.

It is easy to see that $\rhat(K_{1,n},K_{1,m})=n+m-1$
for all positive integers $n$ and $m$. Burr et al.~\cite{BurrEtAl1978} determined
the size Ramsey number of two star forests when all components
within each forest are isomorphic. They conjectured that the
following formula holds for arbitrary star forests.

\begin{conjecture}[Burr et al.~\cite{BurrEtAl1978};
Erd\H{o}s Problem~\#561\protect\footnotemark]
\label{conj:general-star-forest}
Let $s,t$ be positive integers, and let
\[
 n_1\ge\cdots\ge n_s\ge1,
 \qquad
 m_1\ge\cdots\ge m_t\ge1
\]
be sequences of positive integers. Then
\[
\rhat\left(\bigsqcup_{i=1}^{s}K_{1,n_i},
            \bigsqcup_{j=1}^{t}K_{1,m_j}\right)
=\sum_{k=2}^{s+t}\ell_k,
\qquad
\ell_k=\max_{i+j=k}\{n_i+m_j-1\}.
\]
\end{conjecture}
\footnotetext{T.~F. Bloom, Erd\H{o}s Problem \#561,
\texttt{https://www.erdosproblems.com/561}.}


Several cases of Conjecture~\ref{conj:general-star-forest} are known.
Burr et al.~\cite{BurrEtAl1978} proved that
$\rhat(sK_{1,n},tK_{1,m})=(s+t-1)(n+m-1)$ for positive integers
$s,t,n,m$.
Gy\H{o}ri and Schelp~\cite{GyoriSchelp2002} proved that
\[
 \rhat\left(\bigsqcup_{i=1}^{s}K_{1,n_i},
            \bigsqcup_{j=1}^{t}K_{1,m_j}\right)
 =\sum_{k=2}^{s+t}\ell_k
\]
provided that
\[
 \binom{\ell_k}{2}>\sum_{i=k}^{s+t}\ell_i
 \qquad (2\le k\le s+t).
\]
Thus, their result verifies the conjecture when the anti-diagonal terms
satisfy the preceding inequalities.
Cheng~\cite{Cheng2010} studied the case where each star forest is
the disjoint union of a star and a matching.

Davoodi, Javadi, Kamranian and Raeisi~\cite{DavoodiEtAl2025} proved the
one-star case and the two-equal-star case when every component of the
other forest has at least two edges, as well as the case where all
component sizes in both forests are odd.  They also treated a uniform
star forest against a non-uniform star forest under an additional parity
assumption.  More precisely, if
$m_1\ge\cdots\ge m_t\ge2$ and both $n$ and $m_1$ are odd, then
\[
 \rhat\left(sK_{1,n},
            \bigsqcup_{j=1}^{t}K_{1,m_j}\right)
 =(s-1)(n+m_1-1)
  +\sum_{j=1}^{t}(n+m_j-1).
\]

Zhang~\cite{Zhang1992} determined the multicolor size Ramsey numbers
of uniform star forests. Fu, Luo and Ni~\cite{FuLuoNi2026v3}
characterize the size Ramsey minimal graphs for multicolor tuples
of uniform star forests.

We prove Conjecture~\ref{conj:general-star-forest} when one of the
two star forests has exactly two components. This extends the
two-equal-star case of Davoodi et al.~\cite{DavoodiEtAl2025} to two
stars of arbitrary sizes, with an arbitrary star forest as the other
target. In particular, the other forest may contain copies of $K_2$.
No parity condition is imposed.

In the two-component case, write the component sizes as $a\ge b\ge1$
on one side and $n_1\ge\cdots\ge n_q\ge1$ on the other.
Conjecture~\ref{conj:general-star-forest} then takes the form given in
the following theorem.

\begin{theorem}\label{thm:main}
Let $q,a,b,n_1,\ldots,n_q$ be positive integers with
$a\ge b$ and $n_1\ge\cdots\ge n_q$. Then
\begin{equation}\label{eq:main-formula}
\begin{split}
\rhat\left(K_{1,a}\sq K_{1,b},
 \bigsqcup_{i=1}^{q}K_{1,n_i}\right)
={}&(a+n_1-1)
+\sum_{i=1}^{q-1}
 \max\{a+n_{i+1}-1,b+n_i-1\}\\
&+(b+n_q-1).
\end{split}
\end{equation}
When $q=1$, the sum is empty.
\end{theorem}

The upper bound follows from the standard off-diagonal construction.  
For the lower bound, we use lexicographic induction, an iterated
decomposition obtained by deleting matchings, and Lov\'asz's factor theorem.  
When the two stars have different sizes, we also need the boundary estimate
from Claim~\ref{clm:certificate-boundary}.

In Section~2, we collect the preliminary results used in the proof.
In Section~3, we prove the main theorem. In Section~4, we conclude
with some remarks.

\section{Preliminary results}

For a graph \(G\), let \(V(G)\) and \(E(G)\) denote its vertex set and
edge set, respectively, and set \(e(G)=|E(G)|\).
For \(v\in V(G)\), let \(d_G(v)\) be the degree of \(v\) in \(G\),
and let \(\Delta(G)\) be the maximum degree of \(G\).
We write \(\nu(G)\) for the matching number of \(G\), that is,
the maximum size of a matching in \(G\).
For \(X\subseteq V(G)\), let \(G[X]\) denote the subgraph of \(G\)
induced by \(X\), and write \(G-X=G[V(G)\setminus X]\).
In particular, \(G-v=G-\{v\}\).
The symbol \(\sqcup\) denotes disjoint union, and \(tF\) denotes
the disjoint union of \(t\) copies of a graph \(F\).
We write \(K_{1,n}\) for the star with \(n\) edges.

A graph \(G\) is \emph{factor-critical} if \(G-v\) has a perfect
matching for every \(v\in V(G)\).
We will use the following form of Gallai's lemma.

\begin{lemma}[Gallai~\cite{LovaszPlummer1986}]
\label{lem:gallai}
Let $G$ be a connected graph.  If
\[
 \nu(G-v)=\nu(G)\qquad(v\in V(G)),
\]
then $G$ is factor-critical.
\end{lemma}

Since the matching number is additive over components, the condition
$\nu(G-v)=\nu(G)$ for every $v\in V(G)$ implies
$\nu(C-v)=\nu(C)$ for every component $C$ and every $v\in V(C)$.
Thus Lemma~\ref{lem:gallai} applies to each component $C$, so $C$ is
factor-critical and $|V(C)|=2\nu(C)+1$.

For a nonempty non-increasing sequence
$\boldsymbol{x}=(x_1,\ldots,x_r)$ of positive integers, define
 $\mathcal K(\boldsymbol{x})
 =\bigsqcup_{i=1}^{r}K_{1,x_i}.$
If $r\ge2$, write
\[
 \boldsymbol{x}^{+}=(x_2,\ldots,x_r);
 \qquad
 \mathcal K(\boldsymbol{x}^{+})
 =\bigsqcup_{i=2}^{r}K_{1,x_i}.
\]
Thus, $\mathcal K(\boldsymbol{x}^{+})$ is obtained from
$\mathcal K(\boldsymbol{x})$ by deleting its largest star component.

We shall repeatedly use the following vertex-deletion reduction.

\begin{lemma}\label{lem:deletion}
Let $\boldsymbol{x}$ and $\boldsymbol{y}$ be nonempty non-increasing
sequences of positive integers. If
\[
G\to\bigl(\mathcal K(\boldsymbol{x}),\mathcal K(\boldsymbol{y})\bigr),
\]
then, for every $v\in V(G)$,
\[
G-v\to\bigl(\mathcal K(\boldsymbol{x}^{+}),\mathcal K(\boldsymbol{y})\bigr)
\quad\text{and}\quad
G-v\to\bigl(\mathcal K(\boldsymbol{x}),\mathcal K(\boldsymbol{y}^{+})\bigr),
\]
provided that $\boldsymbol{x}^{+}$ and $\boldsymbol{y}^{+}$ are nonempty,
respectively.
\end{lemma}

\begin{proof}
Suppose that $G-v$ admits a coloring with no red
$\mathcal K(\boldsymbol{x}^{+})$ and no blue
$\mathcal K(\boldsymbol{y})$.  Color every edge incident with $v$ red.
If a red copy of $\mathcal K(\boldsymbol{x})$ is obtained, delete its
component containing $v$, or any one component if the copy avoids $v$.
The remaining component sizes, in non-increasing order, are coordinatewise
at least $x_2,\ldots,x_r$.  Deleting surplus leaf edges therefore leaves
a red copy of $\mathcal K(\boldsymbol{x}^{+})$ in $G-v$, a contradiction.
The second assertion follows by interchanging the colors.
\end{proof}

When $n_q\ge2$, Davoodi et al.~\cite{DavoodiEtAl2025} determined the
one-star formula and the corresponding extremal graphs. Their
matching-deletion argument also proves the formula when some components
are copies of $K_2$. 
We include the argument, which will also be used
to prove Claim~\ref{clm:one-star-stability}.

\begin{lemma}\label{lem:one-star}
Let $p,q,n_1,\ldots,n_q$ be positive integers with
$n_1\ge\cdots\ge n_q$.  Then
\[
 \rhat\left(K_{1,p},
             \bigsqcup_{i=1}^{q}K_{1,n_i}\right)
 =\sum_{i=1}^{q}(p+n_i-1).
\]
\end{lemma}

\begin{proof}
Let $H=\bigsqcup_{i=1}^{q}K_{1,n_i}$.
For the upper bound, consider the disjoint union of the stars
$K_{1,p+n_i-1}$, $1\le i\le q$.
In any red-blue coloring with no red $K_{1,p}$, the center of the
$i$th star is incident with at least $n_i$ blue edges.
These blue stars form a copy of $H$, proving the upper bound.

We prove the lower bound by induction on $p$.
If $p=1$ and $G\to(K_2,H)$, then the all-blue coloring of $G$ contains
a copy of $H$, so $e(G)\ge e(H)=\sum_{i=1}^{q}n_i$.
Now let $p\ge2$, assume the assertion holds for $p-1$, and let
$G\to(K_{1,p},H)$.
The all-blue coloring shows that $G$ contains a copy of $H$.
Choose one edge from each component of this copy. These edges form a
matching of size $q$, so a maximum matching $M$ of $G$ satisfies $|M|\ge q$.

Let $G-M$ be the spanning subgraph obtained by deleting the edges of $M$.
We claim that
\[
G-M\to(K_{1,p-1},H).
\]
Suppose to the contrary that \(G-M\) has a red/blue coloring with no red \(K_{1,p-1}\) and no blue copy of \(H\). Extend this coloring to \(G\) by coloring every edge of \(M\) red. Since no blue edge is added, \(G\) still contains no blue copy of \(H\). The red subgraph of \(G-M\) has maximum degree at most \(p-2\). Since \(M\) is a matching, each vertex is incident with at most one edge of \(M\). Hence the red subgraph of \(G\) has maximum degree at most \(p-1\), and therefore contains no copy of \(K_{1,p}\). This contradicts \(G\to(K_{1,p},H)\) and proves the claim.
Applying the induction hypothesis to $G-M$ and using $|M|\ge q$, we obtain
\[
e(G)=e(G-M)+|M|
\ge\sum_{i=1}^{q}(p+n_i-2)+q
=\sum_{i=1}^{q}(p+n_i-1).
\]
This proves the lower bound and completes the proof.
\end{proof}

We shall also use the following elementary special case of the
star-forest conjecture, for which we include a short proof.

\begin{lemma}\label{lem:two-stars-matching}
For positive integers $a\ge b$ and $r$,
\[
 \rhat(K_{1,a}\sq K_{1,b},rK_2)=ra+b.
\]
\end{lemma}

\begin{proof}
For the upper bound, take $Q=rK_{1,a}\sqcup K_{1,b}$.
In a coloring with no blue $rK_2$, at most $r-1$ components of $Q$
contain a blue edge. At least two components are therefore entirely
red, and they contain $K_{1,a}\sqcup K_{1,b}$ since $a\ge b$.
Hence $\rhat(K_{1,a}\sqcup K_{1,b},rK_2)\le e(Q)=ra+b$.

For the lower bound, induct on $r$ and let
$G\to(K_{1,a}\sqcup K_{1,b},rK_2)$.
The all-red coloring shows that $G$ contains $K_{1,a}\sqcup K_{1,b}$,
so $e(G)\ge a+b$ when $r=1$, and $\Delta(G)\ge a$ in general.
For $r\ge2$, choose $v$ of maximum degree.
By Lemma~\ref{lem:deletion}, we have
$G-v\to(K_{1,a}\sqcup K_{1,b},(r-1)K_2)$, so induction yields
\[
e(G)=e(G-v)+d_G(v)\ge(r-1)a+b+a=ra+b.
\]
\end{proof}

For disjoint vertex sets $S,T$, let $e_G(S,T)$ count the edges between
$S$ and $T$, and write $e_G(S)=e(G[S])$.
Let $g,f\colon V(G)\to\mathbb{Z}_{\ge0}$. We write $g\le f$ if
$g(v)\le f(v)$ for every $v\in V(G)$. A \emph{$(g,f)$-factor} of $G$
is a spanning subgraph $J\subseteq G$ such that
\[
g(v)\le d_J(v)\le f(v) \qquad\text{for all }v\in V(G).
\]

We shall use the following nonnegative specialization of Lov\'asz's
$(g,f)$-factor theorem~\cite{Lovasz1970}; see also
\cite{AkiyamaKano2011}.

\begin{lemma}[Lov\'asz's $(g,f)$-factor theorem]
\label{lem:gf-factor}
Let $G$ be a graph, and let
$g,f\colon V(G)\to\mathbb Z_{\ge0}$ satisfy $g\le f$.
Then $G$ has a $(g,f)$-factor if and only if, for every pair of
disjoint sets $S,T\subseteq V(G)$,
\[
 \sum_{v\in S}f(v)
 +\sum_{v\in T}\bigl(d_G(v)-g(v)\bigr)
 -e_G(S,T)-q_{g,f}(S,T)\ge0,
\]
where $q_{g,f}(S,T)$ is the number of components $C$ of
$G-(S\cup T)$ such that $g(v)=f(v)$ for every $v\in V(C)$ and
\[
 \sum_{v\in V(C)}f(v)+e_G(V(C),T)
\]
is odd.
\end{lemma}

For positive integers $\alpha,\beta$, set $D=\alpha+\beta$ and
\begin{equation}\label{eq-3}
A(S,T)=\sum_{v\in T}\min\{d_G(v),\beta\}+\alpha|S|-e_G(S,T).
\end{equation}
Let $q_D(S,T)$ count the components $C$ of $G-(S\cup T)$ in which
every vertex has degree $D$ in $G$ and
$\alpha|V(C)|+e_G(V(C),T)$ is odd.
The following edge partition condition follows from
Lemma~\ref{lem:gf-factor}.

\begin{corollary}\label{lem:lovasz}
Let $G$ be a graph with $\Delta(G)\le D$.
Then $E(G)$ has a partition $E(G)=E(J)\sqcup E(K)$ into spanning
subgraphs satisfying $\Delta(J)\le\alpha$ and $\Delta(K)\le\beta$
if and only if
\[
A(S,T)\ge q_D(S,T)
\]
for all disjoint $S,T\subseteq V(G)$.
\end{corollary}

\begin{proof}
Set $g(v)=\max\{0,d_G(v)-\beta\}$ and $f(v)=\alpha$ for
every $v\in V(G)$. Since $\Delta(G)\le D$, we have $0\le g(v)\le f(v)$.
A spanning subgraph $J$ is a $(g,f)$-factor of $G$ if and only if
$\Delta(J)\le\alpha$ and $\Delta(K)\le\beta$, where $K$ is the
spanning subgraph with $E(K)=E(G)\setminus E(J)$.

Since $\alpha>0$, we have $g(v)=f(v)$ if and only if $d_G(v)=D$.
For each component $C$ of $G-(S\cup T)$, we also have
$\sum_{v\in V(C)}f(v)=\alpha|V(C)|$.
Thus $q_{g,f}(S,T)=q_D(S,T)$ for all disjoint $S,T\subseteq V(G)$.
Moreover,
\[
\begin{aligned}
&\sum_{v\in S}f(v)+\sum_{v\in T}\bigl(d_G(v)-g(v)\bigr)-e_G(S,T)\\
&\qquad=\alpha|S|+\sum_{v\in T}\min\{d_G(v),\beta\}-e_G(S,T)
=A(S,T).
\end{aligned}
\]
The equivalence now follows from Lemma~\ref{lem:gf-factor}.
\end{proof}

A proper edge coloring assigns distinct colors to adjacent edges;
in particular, each color class is a matching.

\begin{lemma}[Vizing~\cite{Vizing1964}]
\label{lem:vizing}
Every simple graph $G$ admits a proper edge coloring with at most
$\Delta(G)+1$ colors.
\end{lemma}

\section{Proof of the main theorem}

\begin{proof}[Proof of Theorem~\ref{thm:main}]
For the upper bound, we use the standard star-forest construction
described by Davoodi et al.~\cite{DavoodiEtAl2025}.
Let $S_1,\ldots,S_{q+1}$ be pairwise disjoint stars with
\begin{align*}
 e(S_1)&=a+n_1-1,
 e(S_i)=\max\{a+n_i-1,b+n_{i-1}-1\}
 \quad(2\le i\le q),\\
 e(S_{q+1})&=b+n_q-1.
\end{align*}
We claim that
\[
\bigsqcup_{i=1}^{q+1}S_i
\to\left(K_{1,a}\sqcup K_{1,b},
         \bigsqcup_{i=1}^{q}K_{1,n_i}\right).
\]
Consider a red/blue coloring of \(\bigsqcup_{j=1}^{q+1}S_j\) with no red copy of \(K_{1,a}\sqcup K_{1,b}\).
If no \(S_j\) with \(1\le j\le q\) contains a red \(K_{1,a}\), then \(e(S_j)\ge a+n_j-1\), so \(S_j\) contains a blue \(K_{1,n_j}\) for each \(1\le j\le q\).
Otherwise, let \(i\) be the smallest index with \(i\le q\) such that \(S_i\) contains a red \(K_{1,a}\). For \(j<i\), the choice of \(i\) gives a blue \(K_{1,n_j}\) in \(S_j\). For \(i<j\le q+1\), the star \(S_j\) contains no red \(K_{1,b}\), since otherwise, together with the red \(K_{1,a}\) in \(S_i\), it would form a red copy of \(K_{1,a}\sqcup K_{1,b}\). Since \(e(S_j)\ge b+n_{j-1}-1\), the star \(S_j\) contains a blue \(K_{1,n_{j-1}}\).

Thus in either case the coloring contains a blue copy of \(\bigsqcup_{j=1}^{q}K_{1,n_j}\). The number of edges in this disjoint union equals the right-hand side of \eqref{eq:main-formula}.

For the lower bound, set
\[
 N=\sum_{i=1}^{q}n_i,
 \qquad
 H=\bigsqcup_{i=1}^{q}K_{1,n_i}.
\]
We use strong induction on $(q,b,N)$ in lexicographic order,
proving each induction step simultaneously for all $a\ge b$ and all
non-increasing sequences $\boldsymbol n=(n_1,\ldots,n_q)$ of positive integers with sum $N$.
For $q=1$, color symmetry and Lemma~\ref{lem:one-star} give the required
value $(a+n_1-1)+(b+n_1-1)$.
For $a=b=1$, let $G\to(2K_2,H)$.
The all-blue coloring contains a copy of $H$, so $\Delta(G)\ge n_1$.
Choose a vertex $v$ of maximum degree. By Lemma~\ref{lem:deletion},
$G-v\to(K_2,H)$, so $G-v$ contains a copy of $H$. Hence
\[
e(G)=d_G(v)+e(G-v)\ge n_1+e(H)=n_1+\sum_{i=1}^{q}n_i.
\]
This is the right side of \eqref{eq:main-formula} when $a=b=1$.
Together with the upper bound proved above, this establishes this case.

Fix \(q\ge2\). Throughout the remainder of the proof, by the induction hypothesis we mean that Theorem~\ref{thm:main} holds for every triple that precedes \((q,b,N)\) in the lexicographic order, and for all admissible choices of \(a\) and \(\boldsymbol n\) corresponding to such a triple.
A \emph{remaining counterexample} is a triple $(a,\boldsymbol n,G)$
at this induction level, with $a\ge b$, $(a,b)\ne(1,1)$ and
$n_1\ge\cdots\ge n_q\ge1$, such that $G\to(K_{1,a}\sqcup K_{1,b},H)$
but $e(G)$ is smaller than the right side of \eqref{eq:main-formula}.
We will show that no such triple exists.

We first establish the residual estimate used after deleting a vertex.

For $q\ge2$, let $\mathcal R(G;a,b;n_1,\ldots,n_q)$ mean that both of
the following Ramsey relations hold:
\begin{equation}\label{eq:residual-one}
 G\to\left(
 K_{1,a}\sqcup K_{1,b},
 \bigsqcup_{i=2}^{q}K_{1,n_i}
 \right),
\end{equation}
and
\begin{equation}\label{eq:residual-two}
 G\to\left(
 K_{1,b},
 \bigsqcup_{i=1}^{q}K_{1,n_i}
 \right).
\end{equation}
When $q=1$, define $\mathcal R(G;a,b;n_1)$ by
$G\to\left(K_{1,b},K_{1,n_1}\right).$
 Let
\[
\begin{split}
 \Lambda_q(a,b;\boldsymbol n)
 ={}&\sum_{i=1}^{q-1}
 \max\{a+n_{i+1}-1,b+n_i-1\}\\
 &+(b+n_q-1).
\end{split}
\]

\begin{claim}\label{clm:residual-bound}
Under this induction hypothesis, if
$\mathcal R(G;a,b;n_1,\ldots,n_q)$ holds, then
\[
 e(G)\ge \Lambda_q(a,b;\boldsymbol n).
\]
\end{claim}

\begin{proof}
In the inductions below, we prove the bound simultaneously for every \(a\ge b\) and every sequence \(n_1\ge\cdots\ge n_q\ge1\).
For $q=1$, the bound $e(G)\ge b+n_1-1=\Lambda_1(a,b;n_1)$
follows from Lemma~\ref{lem:one-star}.
For $q\ge2$, we induct on $b$. To establish the base case $b=1$,
we use an auxiliary induction on the number of blue components,
starting from $q=1$. 
Every application of the main induction
hypothesis in this auxiliary induction uses fewer blue components
than the current main induction step; the residual bound for
$q-1$ components is supplied by the auxiliary induction hypothesis.
If $a+n_2-1\ge n_1$, then \eqref{eq:residual-one} and the main
induction hypothesis for $q-1$ blue components imply
\[
e(G)\ge a+n_2-1+
\sum_{i=2}^{q-1}\max\{a+n_{i+1}-1,n_i\}+n_q
=\Lambda_q(a,1;\boldsymbol n).
\]
Otherwise $n_1>a+n_2-1$.
The all-blue coloring in \eqref{eq:residual-two} shows that
$\Delta(G)\ge n_1$.
Choose a vertex $v$ of maximum degree. 
For \(q\ge3\), applying the second assertion of Lemma~\ref{lem:deletion} to \eqref{eq:residual-one} with
\(\boldsymbol{x}=(a,1)\) and \(\boldsymbol{y}=(n_2,\ldots,n_q)\), and to \eqref{eq:residual-two} with
\(\boldsymbol{x}=(1)\) and \(\boldsymbol{y}=(n_1,\ldots,n_q)\), we obtain the residual relation
\[
\mathcal R(G-v;a,1;n_2,\ldots,n_q).
\]
For \(q=2\), only the one-star relation in this definition is needed.
Induction on $q$ therefore yields
\[
e(G)\ge n_1+\Lambda_{q-1}(a,1;n_2,\ldots,n_q)
=\Lambda_q(a,1;\boldsymbol n).
\]

Now let $b\ge2$, and choose a maximum matching $M$ of $G$.
Since \eqref{eq:residual-two} forces a copy of $H$, we have $|M|\ge q$.
Extend any coloring of $G-M$ by coloring $M$ red.
Each red star loses at most one edge when $M$ is deleted, while all
blue edges remain in $G-M$.
Thus both residual relations hold with $a,b$ replaced by $a-1,b-1$,
that is, $\mathcal R(G-M;a-1,b-1;n_1,\ldots,n_q)$ holds.
Each of the $q$ terms defining $\Lambda_q$ decreases by one when
$a,b$ are both decreased by one. Using induction on $b$ and $|M|\ge q$, we obtain
\[
e(G)=e(G-M)+|M|
\ge\Lambda_q(a-1,b-1;\boldsymbol n)+q
=\Lambda_q(a,b;\boldsymbol n).
\]
\end{proof}

For fixed parameters, let
\begin{equation}\label{eq-2}
\begin{aligned}
L &= (a+n_1-1)+\Lambda_q(a,b;\boldsymbol n)\\
  &= (a+n_1-1)
   +\sum_{i=1}^{q-1}
    \max\{a+n_{i+1}-1,\; b+n_i-1\}
   +(b+n_q-1).
\end{aligned}
\end{equation}
With $L$ fixed as above, the following claim shows that any graph
satisfying the displayed Ramsey relation and $e(G)\le L-1$ must have
$e(G)=L-1$ and $\Delta(G)=a+n_1-2$. Moreover, it admits no edge
partition $E(G)=E(J)\sqcup E(K)$ with $\Delta(J)\le a-1$ and
$\Delta(K)\le n_1-1$.

\begin{claim}\label{clm:first-step}
Under this induction hypothesis, if
\[
 G\to\left(K_{1,a}\sq K_{1,b},
            \bigsqcup_{i=1}^{q}K_{1,n_i}\right)
\]
and $e(G)\le L-1$, then
\[
 e(G)=L-1,
 \qquad
 \Delta(G)=a+n_1-2.
\]
Moreover, $G$ has no edge partition into a graph of maximum degree at
most $a-1$ and a graph of maximum degree at most $n_1-1$.
\end{claim}

\begin{proof}
In particular,
$G\to(K_{1,a},K_{1,n_1}).$
Indeed, any coloring with neither a red $K_{1,a}$ nor a blue
$K_{1,n_1}$ has neither a red
$K_{1,a}\sqcup K_{1,b}$ nor a blue
$\bigsqcup_{i=1}^{q}K_{1,n_i}$.

We first prove that
\[
 \Delta(G)\ge a+n_1-2.
\]
Suppose instead that $\Delta(G)\le a+n_1-3$. By
Lemma~\ref{lem:vizing}, the edges of $G$ admit a proper coloring
with at most
\[
 \Delta(G)+1\le a+n_1-2
\]
colors. Add empty color classes if necessary so that there are exactly
$a+n_1-2$ classes. 
Color the edges in $a-1$ classes red and those in
the remaining $n_1-1$ classes blue. Because each class is a matching,
the red graph has maximum degree at most $a-1$, while the blue graph
has maximum degree at most $n_1-1$. This coloring contains neither a
red $K_{1,a}$ nor a blue $K_{1,n_1}$, a contradiction.

Choose a vertex $v$ with $d_G(v)=\Delta(G)$. 
By Lemma~\ref{lem:deletion}, deleting $v$ from the original arrowing
relation gives both relations in
$\mathcal R(G-v;a,b;n_1,\ldots,n_q)$. 
Thus, by Claim~\ref{clm:residual-bound}, we have
\[
e(G-v) \ge \Lambda_q(a,b;\boldsymbol n).
\]
If $\Delta(G)\ge a+n_1-1$, then
\[
\begin{aligned}
 e(G)
 &=d_G(v)+e(G-v)\\
 &\ge(a+n_1-1)+\Lambda_q(a,b;\boldsymbol n)=L,
\end{aligned}
\]
contrary to $e(G)\le L-1$. Therefore we have
$\Delta(G)=a+n_1-2.$
Substituting this equality into
$e(G)=d_G(v)+e(G-v)$ implies
\[
 e(G)\ge(a+n_1-2)+\Lambda_q(a,b;\boldsymbol n)\overset{(\ref{eq-2})}{=}L-1.
\]
Together with the assumed upper bound, this proves $e(G)=L-1$.

Finally, suppose that
\[
 E(G)=E(G_1)\sqcup E(G_2),
 \qquad
 \Delta(G_1)\le a-1,
 \qquad
 \Delta(G_2)\le n_1-1.
\]
Color $G_1$ red and $G_2$ blue. The resulting coloring contains
neither a red $K_{1,a}$ nor a blue $K_{1,n_1}$, and hence contains
neither of the two target forests. This contradicts the arrowing
relation and proves the last assertion.
\end{proof}

The number of edges and the maximum degree of a remaining counterexample are determined by Claim~\ref{clm:first-step}.
Using these conclusions, the next claim rules out \(b\ge2\), so it remains to consider \(b=1\).

\begin{claim}\label{clm:reduction-b-one}
Under this induction hypothesis, no remaining
counterexample has \(b\ge2\).
\end{claim}

\begin{proof}
Suppose that $b\ge2$ and $G$ is a remaining counterexample.
By Claim~\ref{clm:first-step}, $e(G)=L-1$.
A maximum matching $M$ has at least $q$ edges because $G$ contains $H$.
By the red matching-extension argument in Claim~\ref{clm:residual-bound},
$G-M\to(K_{1,a-1}\sqcup K_{1,b-1},H)$.
The induction hypothesis at $(q,b-1,N)$ implies $e(G-M)\ge L-(q+1)$,
since all $q+1$ terms of the target formula decrease by one.
Consequently $|M|\le q$, and hence $\nu(G)=q$.

The following component argument is valid for every $b\ge1$ when
$\nu(G)=q$.
For every \(v\in V(G)\), Lemma~\ref{lem:deletion} yields
\[
G-v\to(K_{1,b},H).
\]
Hence \(G-v\) contains a copy of \(H\), and \(\nu(G-v)=q=\nu(G)\).
Deleting the isolated vertices of $G$ does not affect these matching
numbers. Applying Lemma~\ref{lem:gallai} to each nontrivial component,
as explained after the lemma, shows that the nontrivial components
$C_1,\ldots,C_s$ of $G$ are factor-critical.
Writing $r_j=\nu(C_j)$, we have $r_j\ge1$,
$\sum_jr_j=q$ and $|V(C_j)|=2r_j+1$.
We claim that, for every $j$,
\[
C_j\to(K_{1,a}\sqcup K_{1,b},r_jK_2).
\]
Suppose, to the contrary, that some $C_j$ has a red-blue coloring with no
red $K_{1,a}\sqcup K_{1,b}$ and no blue matching of size $r_j$. Extend this
coloring to $G$ by coloring every edge outside $C_j$ blue. The resulting
coloring has no red copy of $K_{1,a}\sqcup K_{1,b}$. Its blue matching
number is at most
\[
(r_j-1)+\sum_{\ell\ne j}r_\ell=q-1,
\]
because the blue edges inside $C_j$ contain no matching of size $r_j$,
and the matching number of each other component $C_\ell$ is $r_\ell$.
On the other hand, every copy of $H=\bigsqcup_{i=1}^{q}K_{1,n_i}$
contains a matching of size $q$, obtained by choosing one edge from each
star component. Thus the coloring contains no blue copy of $H$, contrary to
$G\to(K_{1,a}\sqcup K_{1,b},H)$. This contradiction proves the claim.
Applying Lemma~\ref{lem:two-stars-matching} to each $C_j$ and summing,
we obtain
\[
e(G)\ge qa+sb.
\]

Set $x_i=n_i-1$.
Fix a copy of $H$, and for each $j$, let $I_j$ be the set of indices
of the star components of this copy that are contained in $C_j$. For every
$i\in I_j$, choose one edge from the corresponding star. These chosen
edges form a matching in $C_j$, because the star components of $H$ are
vertex-disjoint. Consequently, $|I_j|\le r_j$. Every star of $H$ is
contained in exactly one component $C_j$, so $\sum_{j=1}^{s}|I_j|=q$.
Since $\sum_{j=1}^{s}r_j=q$, it follows that $|I_j|=r_j$ for every $j$.
Since these stars are vertex-disjoint, we have
\[
2r_j+\sum_{i\in I_j}x_i
=\sum_{i\in I_j}(n_i+1)\le |V(C_j)|=2r_j+1.
\]
Thus every $x_i$ belongs to $\{0,1\}$.
Thus $G$ has $2q+s$ nonisolated vertices and $H$ has
$2q+\sum_{i=1}^{q}x_i$ vertices. For a nonisolated $v$, $G-v$ contains $H$
and has at most $2q+s-1$ nonisolated vertices; since $H$ has no isolated
vertices, $2q+\sum_{i=1}^{q}x_i\le 2q+s-1$, whence
\begin{equation}\label{eq:reduction-b-excess-bound}
\sum_{i=1}^{q}x_i\le s-1\le q-1.
\end{equation}

Returning to $b\ge2$, we have $\max\{a+x_{i+1},b+x_i\}\le a+x_i$
because $a\ge b$ and $x_i\ge x_{i+1}$.
Using \eqref{eq:main-formula}, $x_1\le\sum_{i=1}^{q}x_i$ and
\eqref{eq:reduction-b-excess-bound}, we obtain
\[
L\le qa+b+\sum_{i=1}^{q}x_i+x_1
\le qa+b+2(s-1)
\le qa+sb\le e(G),
\]
contrary to $e(G)=L-1$.
\end{proof}

By Claim~\ref{clm:reduction-b-one}, it remains only to exclude
counterexamples with \(b=1\) and \(a\ge2\). 
Recall that $x_i=n_i-1$.
For convenience, set
\[
\rho=\sum_{i=1}^{q-1}\max\{n_i-n_{i+1}-a+1,0\}.
\]
The next claim analyzes this remaining case. 
It shows that the smallest star component of \(H\) is a single edge, gives an upper bound on \(\nu(G)-q\), and determines \(e(G)\) exactly.

\begin{claim}
\label{clm:remaining-matching}
Assume that \(q\ge2\), \(a\ge2\), \(b=1\), and that $G$ is a
remaining counterexample. Then \(n_q=1\), and
\[
\nu(G)=q+k
\]
for some integer \(k\) satisfying
$1\le k\le\rho.$
Moreover, we have
\[
e(G)=qa+\sum_{i=1}^{q}x_i+\rho.
\]
\end{claim}

\begin{proof}
Since the all-blue coloring contains $H$, we have $\nu(G)\ge q$.
If equality holds, apply the component argument in
Claim~\ref{clm:reduction-b-one} with $b=1$.
Then $e(G)\ge qa+s$, $x_i\in\{0,1\}$ and
$\sum_{i=1}^{q}x_i\le s-1\le q-1$.
Thus $x_q=0$ and $x_i-x_{i+1}\le1\le a-1$.
The target formula reduces to
$L=qa+1+\sum_{i=1}^{q}x_i\le qa+s\le e(G)$,
a contradiction. Hence $\nu(G)\ge q+1$.

If $n_q\ge2$, take a maximum matching $M$ and extend colorings of
$G-M$ by coloring $M$ blue.
Each blue star loses at most one edge, so
\[
G-M\to\left(K_{1,a}\sqcup K_2,
\bigsqcup_{i=1}^{q}K_{1,n_i-1}\right).
\]
The smaller positive sequence has sum $N-q$, so the induction
hypothesis at $(q,1,N-q)$ implies $e(G-M)\ge L-(q+1)$.
Together with $|M|\ge q+1$, this implies $e(G)\ge L$, a contradiction.
Therefore $n_q=1$.

Write $\nu(G)=q+k$, where $k\ge1$, and again take a maximum matching $M$.
Extending a coloring of $G-M$ by coloring $M$ red shows that
$G-M\to(K_{1,a-1},H)$: a red $K_{1,a}$ in the extended coloring
retains at least $a-1$ edges in $G-M$.
By Lemma~\ref{lem:one-star} and $\nu(G)=q+k$, we have
\[
e(G)\ge |M|+q(a-1)+\sum_{i=1}^{q}x_i
=qa+\sum_{i=1}^{q}x_i+k.
\]
Since $x_q=0$, the identity
$\max\{a+x_{i+1},1+x_i\}=a+x_{i+1}
+\max\{x_i-x_{i+1}-a+1,0\}$
reduces  (\ref{eq-2}) to
$L=qa+\sum_{i=1}^{q}x_i+\rho+1$.
By Claim~\ref{clm:first-step},
\(e(G)=L-1=qa+\sum_{i=1}^{q}x_i+\rho\).
Comparing this with the lower bound above yields \(k\le\rho\).
\end{proof}

Iterating the matching deletion argument in the proof of Lemma~\ref{lem:one-star}, we obtain the following decomposition for graphs whose number of edges exceeds the one-star bound by \(h\).
Let \(p,q,n_1,\ldots,n_q\) be positive integers with
$n_1\ge\cdots\ge n_q$, and write
$H=\bigsqcup_{i=1}^{q}K_{1,n_i}$.
\begin{claim}
\label{clm:one-star-stability}

Suppose that
$
 Q\to(K_{1,p},H)
~\text{and}~
 e(Q)=\sum_{i=1}^{q}(p+n_i-1)+h
$
for some integer $h\ge0$.  Then there exist matchings
$M_1,\ldots,M_{p-1}$ in $Q$ and a subgraph $B\subseteq Q$
containing a copy of $H$ such that
\[
 E(Q)=E(B)\sqcup E(M_1)\sqcup\cdots\sqcup E(M_{p-1})
\]
and
\[
 e(B)-e(H)\le h.
\]
When $p=1$, the family of matchings is empty.
\end{claim}

\begin{proof}
If $p=1$, take $B=Q$. Since $Q\to(K_2,H)$, the all-blue coloring
contains a copy of $H$ in $B$. Also, $e(B)-e(H)=h$ by the assumed edge count.
Assume that $p\ge2$ and set $Q_0=Q$.
For $1\le i\le p-1$, choose a maximum matching $M_i$ in $Q_{i-1}$
and let $Q_i=Q_{i-1}-M_i$.
Applying the matching-deletion argument in the proof of
Lemma~\ref{lem:one-star} at each step, we obtain
\[
Q_i\to(K_{1,p-i},H)\qquad(0\le i\le p-1).
\]
Thus each $Q_i$ contains a copy of $H$.
In particular, choosing one edge from each component of a copy of $H$
in $Q_{i-1}$ shows that $|M_i|\ge q$ for $1\le i\le p-1$.

Set $B=Q_{p-1}$.
Then $B$ contains a copy of $H$, and by construction
\[
E(Q)=E(B)\sqcup E(M_1)\sqcup\cdots\sqcup E(M_{p-1}).
\]
Using $|M_i|\ge q$ for every $i$ and the assumed value of $e(Q)$,
we obtain
\[
e(B)=e(Q)-\sum_{i=1}^{p-1}|M_i|
\le\sum_{j=1}^{q}(p+n_j-1)+h-(p-1)q
=e(H)+h,
\]
as required.
\end{proof}

We next obtain the edge decomposition needed in the final argument.
We continue to use the quantity $\rho$ defined above.

\begin{claim}\label{clm:unit-tail}
Let \(G\) be a remaining counterexample, and let
$\nu(G)=q+k.$
Then there exist matchings \(M_1,\ldots,M_{a-1}\) in \(G\), a
subgraph $H_0\cong H=\bigsqcup_{i=1}^{q}K_{1,n_i},$
and a subgraph \(R\subseteq G\) such that their edge sets form the
following partition of \(E(G)\):
\begin{equation}\label{eq:stability-decomposition}
E(G)
 =E(M_1)\sqcup\cdots\sqcup E(M_{a-1})
  \sqcup E(H_0)\sqcup E(R).
\end{equation}
Moreover,
\begin{equation}\label{eq:remainder-bound}
e(R)\le \rho-k.
\end{equation}
\end{claim}

\begin{proof}
Let $M$ be a maximum matching.
By Claim~\ref{clm:remaining-matching} and its proof,
$|M|=q+k$, $G-M\to(K_{1,a-1},H)$, and
\[
e(G-M)=\sum_{i=1}^{q}(a+n_i-2)+(\rho-k).
\]
Apply Claim~\ref{clm:one-star-stability} to $G-M$ with
$p=a-1$ and $h=\rho-k\ge0$.
There are matchings $M_1,\ldots,M_{a-2}$ and a remaining subgraph
$B$ containing a copy $H_0$ of $H$, with $e(B)-e(H_0)\le\rho-k$.
Taking $M_{a-1}=M$ and $E(R)=E(B)\setminus E(H_0)$ gives the claimed
partition and $e(R)\le\rho-k$.
When $a=2$, the family $M_1,\ldots,M_{a-2}$ is empty.
\end{proof}

Fix the decomposition in Claim~\ref{clm:unit-tail}, and denote the
centers of $H_0$ by $c_1,\ldots,c_q$, with $d_{H_0}(c_i)=n_i$.
The small number of edges in $R$ restricts the positive differences
between consecutive star sizes.

\begin{claim}\label{clm:positive-gaps}
At least two indices \(i\in\{1,\ldots,q-1\}\) satisfy
$n_i-n_{i+1}-a+1>0.$
\end{claim}

\begin{proof}
For each $i$, the graph $G$ has at least $i+1$ vertices of degree at
least $n_i$. Indeed, $d_G(c_i)\ge n_i$. By
Lemma~\ref{lem:deletion}, $G-c_i\to(K_2,H)$, so the all-blue coloring
of $G-c_i$ contains a copy of $H$. The centers of the $i$ largest
stars in this copy have degree at least $n_i$ in $G-c_i$, since
$n_1\ge\cdots\ge n_i$. These $i$ distinct centers are different from
$c_i$; together with $c_i$, they give $i+1$ distinct vertices of
degree at least $n_i$ in $G$.

If $n_i-n_{i+1}-a+1>0$, choose a vertex $v\notin\{c_1,\ldots,c_i\}$
with $d_G(v)\ge n_i$.
Then $d_{H_0}(v)\le n_{i+1}$, and the \(a-1\) matchings contribute
at most \(a-1\) to its degree. Hence
\[
n_i-n_{i+1}-a+1
\le d_R(v)\le e(R)\le\rho-k<\rho.
\]
Since
\[
\rho=\sum_{i=1}^{q-1}\max\{n_i-n_{i+1}-a+1,0\}\ge k\ge1,
\]
there must be at least two positive summands.
\end{proof}

If \(q=2\), then the index set \(\{1,\ldots,q-1\}\) contains only
the index \(i=1\). This contradicts Claim~\ref{clm:positive-gaps},
which requires at least two indices satisfying
\(n_i-n_{i+1}-a+1>0\). Hence every remaining counterexample satisfies
\(q\ge3\). We now identify the vertices of maximum degree.
Put \(D=a+n_1-2\).

\begin{claim}
\label{clm:max-degree-centers}
With this notation, \(\Delta(G)=D\), and every vertex of degree \(D\) is a center of
the fixed copy \(H_0\).
\end{claim}

\begin{proof}
By Claim~\ref{clm:first-step}, $\Delta(G)=a+n_1-2=D$.
Since $n_q=1$ and Claim~\ref{clm:positive-gaps} gives at least two
indices with $n_i-n_{i+1}\ge a$,
\begin{align*}
n_1-1-\rho
&=\sum_{i=1}^{q-1}\left(n_i-n_{i+1}
-\max\{n_i-n_{i+1}-a+1,0\}\right)\\
&=\sum_{i=1}^{q-1}\min\{n_i-n_{i+1},a-1\}
\ge2(a-1)>0.
\end{align*}
Consequently, we have \(a+\rho-1<D\).
For $v\notin\{c_1,\ldots,c_q\}$, we have $d_{H_0}(v)\le1$.
Using \eqref{eq:stability-decomposition} and \eqref{eq:remainder-bound},
we obtain
\[
d_G(v)\le a+e(R)
\le a+\rho-k\le a+\rho-1<D.
\]
Thus every vertex of maximum degree is a center of $H_0$.
\end{proof}

We next show that $a\ge3$. Suppose that $a=2$.
Then $D=a+n_1-2=n_1$.
Let \(W\) be the set of vertices of degree \(D\). By
Claim~\ref{clm:max-degree-centers}, every vertex of \(W\) is a center
of the fixed copy \(H_0\). 
Select one edge from each component of $H_0$ whose center lies in $W$.
Since the star components are vertex-disjoint, the selected edges
form a matching $M$ that covers $W$.

Color the edges of \(M\) red and all remaining edges blue. The red
graph has maximum degree at most \(1=a-1\). Every vertex in \(W\)
loses an incident edge when \(M\) is removed, whereas every vertex
outside \(W\) already has degree at most \(D-1\). Hence the blue graph
has maximum degree at most
$D-1=n_1-1.$
This bounded-degree edge partition contradicts
Claim~\ref{clm:first-step}.

Hence $a\ge3$, and therefore $a-1\ge2$.
Set $\beta=n_1-1$ and $\alpha=a-1$. Since \(D=a+n_1-2\), 
we have
$D=\alpha+\beta$.
We use $A(S,T)$ and $q_D(S,T)$ as defined before Corollary~\ref{lem:lovasz}
to prove the required edge partition.

The next estimate bounds the edges from $S\cup T$ to its complement
in terms of $A(S,T)$.

\begin{claim}
\label{clm:certificate-boundary}
Suppose
$\Delta(G)\le D$. For disjoint sets $S,T\subseteq V(G)$ and
$Y=V(G)\setminus(S\cup T)$, the following hold:
\begin{equation}
\label{eq:A-nonnegative}
A(S,T)\ge 0,
\end{equation}
and
\begin{align}
e_G(S,Y)
 &\le \frac{D}{\alpha}\,A(S,T),
\label{eq:S-boundary}\\
e_G(T,Y)
 &\le \frac{D}{\beta}\,A(S,T).
\label{eq:T-boundary}
\end{align}
\end{claim}

\begin{proof}
Since $0\le d_G(v)\le D=\alpha+\beta$, we have
$\min\{d_G(v),\beta\}\ge\frac{\beta}{D}d_G(v)$
for $v\in V(G).$
Also, the sum of the degrees over $S$ satisfies
\[
D|S|\ge\sum_{v\in S}d_G(v)
=2e_G(S)+e_G(S,T)+e_G(S,Y).
\]
Using these inequalities and
$\sum_{v\in T}d_G(v)=2e_G(T)+e_G(S,T)+e_G(T,Y)$, we obtain
\begin{equation}\label{eq-1}
\begin{aligned}
DA(S,T)
&\overset{(\ref{eq-3})}{\ge}\beta\sum_{v\in T}d_G(v)+\alpha D|S|-D e_G(S,T)\\
&=\beta\bigl(2e_G(T)+e_G(T,Y)\bigr)
  +\alpha\bigl(D|S|-e_G(S,T)\bigr)\\
&\ge\alpha e_G(S,Y)+\beta e_G(T,Y)\ge0.
\end{aligned}
\end{equation}
Since $\alpha,\beta>0$, this proves \eqref{eq:A-nonnegative},
\eqref{eq:S-boundary}, and \eqref{eq:T-boundary}.
\end{proof}

We now apply Claim~\ref{clm:certificate-boundary} and
Corollary~\ref{lem:lovasz} to obtain the required edge partition.
Recall that $q_D(S,T)$ count the components $C$ of $G-(S\cup T)$ in which
every vertex has degree $D$ in $G$ and
$\alpha|V(C)|+e_G(V(C),T)$ is odd.

\begin{claim}\label{clm:bounded-partition}
The graph \(G\) has an edge partition
\[
E(G)=E(J)\sqcup E(K)
\]
such that
\[
\Delta(J)\le \alpha
\qquad\text{and}\qquad
\Delta(K)\le \beta.
\]
\end{claim}

\begin{proof}
Suppose that no such partition exists.
Since $\Delta(G)=D$, Corollary~\ref{lem:lovasz} implies that
$A(S,T)<q_D(S,T)$ for some disjoint sets $S,T$.
Put $r=q_D(S,T)$ and $Y=V(G)\setminus(S\cup T)$.
By Claim~\ref{clm:certificate-boundary} and integrality, we have
\begin{equation}\label{eq:A-defect}
0\le A(S,T)\le r-1,
\end{equation}
so $r\ge1$.
Since $n_q=1$, and Claim~\ref{clm:positive-gaps} ensures that
at least two indices satisfy $n_i-n_{i+1}\ge a$, we have
\begin{equation}\label{eq:beta-large}
\beta=n_1-1=\sum_{i=1}^{q-1}(n_i-n_{i+1})\ge2a.
\end{equation}
By \eqref{eq:beta-large} and $a\ge3$, we have
\(\beta\ge2a\ge a-1=\alpha\). Hence \eqref{eq-1} gives
\[
\begin{aligned}
\alpha\bigl(e_G(S,Y)+e_G(T,Y)\bigr)
&\le\alpha e_G(S,Y)+\beta e_G(T,Y)\\
&\le DA(S,T).
\end{aligned}
\]
Dividing by $\alpha>0$ and using \eqref{eq:A-defect}, we obtain
\begin{equation}\label{eq:total-boundary}
e_G(S,Y)+e_G(T,Y)
\le\frac D{\alpha} A(S,T)\le\frac D{\alpha}(r-1).
\end{equation}

Consider the $r$ components counted by $q_D(S,T)$.
Every vertex of each such component has degree $D$ in $G$.
If every such component has at most \(D\) vertices, then for each such component \(C\),
\[
e_G(V(C),S\cup T)=D|V(C)|-2e(C)
\ge |V(C)|(D-|V(C)|+1)\ge D.
\]
Summing over these $r$ components and using
\eqref{eq:total-boundary} and $\alpha\ge2$,
we obtain
\[
rD\le \frac{D(r-1)}{\alpha}\leq D(r-1)<rD,
\]
a contradiction.
Thus at least one of these components has at least $D+1$ vertices.
Since each of the other $r-1$ components has at least one vertex,
the $r$ components contain at least $D+r$ vertices in total.

By Claim~\ref{clm:max-degree-centers}, every vertex in these components
is a center of $H_0$.
Every leaf of a star whose center lies in one of these components belongs to \(S\cup T\). Indeed, otherwise it would lie in the same counted component as its center, would have degree \(D\), and hence would also be a center, a contradiction.
For each center $c_i$ in these components, the edge partition
\eqref{eq:stability-decomposition} in Claim~\ref{clm:unit-tail} gives
\[
d_G(c_i)=\sum_{j=1}^{a-1}d_{M_j}(c_i)+d_{H_0}(c_i)+d_R(c_i).
\]
Each $M_j$ is a matching, so $d_{M_j}(c_i)\le1$.
Since $d_{H_0}(c_i)=n_i$ and $d_G(c_i)=D$, we obtain
$D\le a-1+n_i+d_R(c_i)$.
Thus $n_i\ge D-(a-1)-d_R(c_i)=\beta-d_R(c_i)$.
There are at least $D+r$ such centers. Their corresponding stars are
vertex-disjoint, with centers in $Y=V(G)\setminus(S\cup T)$ 
and leaves in \(S\cup T\).
Hence their edges contribute \(\sum n_i\) distinct edges to
\(e_G(S,Y)+e_G(T,Y)\), where the sum is over these centers.

Summing \(n_i\ge\beta-d_R(c_i)\) over these centers,
and using
that there are at least \(D+r\) of them and that
\(\sum d_R(c_i)\le2e(R)\), we obtain
\begin{equation}\label{eq:boundary-lower}
e_G(S,Y)+e_G(T,Y)
\ge\beta(D+r)-2e(R).
\end{equation}
For each \(i\), we have
\[
\max\{n_i-n_{i+1}-a+1,0\}\le n_i-n_{i+1}.
\]
Thus, by the definition of \(\rho\) and \(n_q=1\),
\[
\rho\le\sum_{i=1}^{q-1}(n_i-n_{i+1})
=n_1-n_q=n_1-1=\beta.
\]
Since \(k\ge1\), it follows from \eqref{eq:remainder-bound} that
$e(R)\le \rho-k\le \beta-1.$
Therefore \eqref{eq:boundary-lower} implies
\[
e_G(S,Y)+e_G(T,Y)\ge\beta(D+r-2)+2>\beta(r-1).
\]
On the other hand, $D/\alpha=1+\beta/\alpha\le1+\beta/2<\beta$
because $\alpha=a-1\ge2$ and $\beta\ge6$ by \eqref{eq:beta-large}.
Together with \eqref{eq:total-boundary}, this yields
$e_G(S,Y)+e_G(T,Y)\le\beta(r-1)$, a contradiction.
\end{proof}

By Claim~\ref{clm:bounded-partition}, there exist subgraphs $J$ and $K$ such that
\[
E(G)=E(J)\sqcup E(K),\qquad
\Delta(J)\le \alpha=a-1,\qquad
\Delta(K)\le \beta=n_1-1.
\]
However, Claim~\ref{clm:first-step} rules out exactly such an edge partition in a remaining counterexample.
Hence no remaining counterexample exists at this
induction step. The base case has already been verified, and the same argument
applies to all admissible triples. 
Thus Theorem~\ref{thm:main} follows by
strong induction.
\end{proof}

\section{Concluding remarks}

We have proved the conjectured formula in the case where one of the two star forests has exactly two components. Lemma~\ref{lem:one-star} gives a lower bound for the relation obtained from Lemma~\ref{lem:deletion} whose target is a single star. The proof uses both vertex deletion relations, together with induction, matching arguments, and factor arguments.

Suppose that both star forests have at least three components.
The vertex deletion relations in Lemma~\ref{lem:deletion} then
have at least two components in each target. If one of these
targets has exactly two components, then Theorem~\ref{thm:main}
gives the corresponding size Ramsey number. It remains to extend
the lower bound argument to arbitrary pairs of star forests each
having at least three components.

\bigskip\noindent
\textbf{Declaration on the use of AI.}
The author used ChatGPT to assist in exploring proof strategies and improving exposition. All mathematical arguments, results, and conclusions were verified by the author, who takes full responsibility for them.

\end{document}